\documentclass[12pt,reqno]{amsart}
\usepackage{amssymb,amsmath,graphicx,
	amsfonts,euscript}
\usepackage{color}
\usepackage{mathrsfs}

\newcommand{\norma}[2]{\left\lVert #1 \right\rVert_{#2}}

\def\nn{\nonumber}
\def\div{ \hbox{\rm div}\,  }
\theoremstyle{plain}
\newcommand{\p}{{\mathbb P}}
\newcommand{\q}{{\mathbb Q}}
\def\Re{\mathrm{Re}}

\newcommand{\jap}[1]{\left\langle #1 \right\rangle}
\newcommand{\norm}[1]{\left\lVert #1 \right\rVert}
\newcommand{\R}{\mathbb{R}}
\newcommand{\T}{\mathbb{T}}

\theoremstyle{plain}

\newtheorem{theorem}{Theorem}[section]
\newtheorem{lemma}[theorem]{Lemma}

\newtheorem{remark}[theorem]{Remark}
\numberwithin{equation}{section}

\begin{document}

\title[stability analysis of the compressible fluid ]{Linear stability of the   Couette flow for the  non-isentropic compressible fluid}

\author[Xiaoping Zhai]{Xiaoping Zhai}

\address[X. Zhai]{ School of Mathematics and Statistics, Shenzhen University, Shenzhen, 518060, China.} \email{zhaixp@szu.edu.cn }

\begin{abstract}
We are concerned with  the
 linear stability  of the   Couette flow for the   non-isentropic compressible Navier-Stokes equations with vanished shear viscosity in a domain $\mathbb{T}\times \mathbb{R}$. For a general initial data settled in Sobolev spaces, we obtain a Lyapunov type instability of the  density, the temperature, the compressible part of the velocity field, and also obtain  an inviscid damping for the incompressible part of the velocity field. Moreover, if the initial  density, the initial temperature and the incompressible part of the initial velocity field satisfy some quality relation, we can prove the enhanced dissipation phenomenon for the  velocity field.
	\end{abstract}
\maketitle

\section{ Introduction and the main result}
In this paper, we are interested in the long-time asymptotic behaviour of the linearized two dimensional  non-isentropic compressible Navier-Stokes equations in a domain $\mathbb{T}\times \mathbb{R}$. The governing equations (in non-dimensional variables) are
\begin{eqnarray}\label{mmma}
\left\{\begin{aligned}
&\varrho_t+{\mathbf{u}}\cdot\nabla {\varrho}+{\varrho}\div{\mathbf{u}}=0,\\
&{{\varrho}}({\mathbf{u}_t}+{\mathbf{u}}\cdot\nabla {\mathbf{u}})
+\frac{1}{\gamma  {M}^2}\nabla P=\frac{1}{\Re}\left(\mu\Delta\mathbf{u}+(\nu+\mu)\nabla \div\mathbf{u}\right),\\
&{{\varrho}}({\vartheta_t}+{\mathbf{u}}\cdot\nabla \vartheta)
+ (\gamma-1)P\div {\mathbf{u}}=\!\frac{\gamma\mu}{\sigma \Re}\Delta\vartheta+\!\frac{\gamma(\gamma-1)M^2}{\Re}\left(\frac{\mu}{2}|\nabla {\mathbf u}+\!\nabla {\mathbf u}^\top|^2+\nu|\div {\mathbf u}|^2\right).
\end{aligned}\right.
\end{eqnarray}
Here $t\ge0$ is time, $(x,y)\in\mathbb{T}\times \mathbb{R}$ is the spatial coordinate and $\mathbb{T}=\mathbb{R}/\mathbb{Z}$.
 The unknown ${\mathbf{u}}$ is the velocity vector, ${{\varrho}}$ is the density, $\vartheta$ is the
temperature, $P={\varrho}\vartheta$ is the pressure.  $\gamma>1$ is the ratio of specific heats, $M>0$ is the Mach number of the reference state, $\Re>0$ is the Reynolds number,  and $\sigma>0$ is the Prandtl number. The two constant viscosity coefficients
$\mu$  and  $\nu$ are the shear viscosity and the  volume viscosity respectively.
The equations \eqref{mmma} then express respectively the conservation of mass, the
balance of momentum, and the balance of energy under internal pressure, viscosity
forces, and the conduction of thermal energy.

A comprehensive understanding of the stability of compressible or incompressible shear flows is a
fundamental problem in fluid mechanics and has been the subject of both theoretical
and practical interest in astrophysics and engineering,
see \cite{dole}, \cite{doln}, \cite{abc12+1}--\cite{cha2}, \cite{dra}--\cite{abc19}, \cite{mal}, \cite{mar}--\cite{abcde11}, \cite{zenglan}, \cite{zhaixiaoping} for the compressible fluid
and  \cite{bed2017}--\cite{abc12}, \cite{abc15}--\cite{dengwen}, \cite{abcde10}, \cite{mas}, \cite{wei1}--\cite{abcde12}
for incompressible fluid.
The aim of the present paper is to study the long-time asymptotic behaviour of the linearized   non-isentropic compressible Navier-Stokes equations around the  Couette flow. That is we  seek a stationary solution of \eqref{mmma} with a constant mean pressure which have the following form:
\begin{align}\label{m2}
\varrho_{sh}=\varrho_{sh}(y),\quad{\mathbf{u}}_{sh}=\begin{pmatrix}
				y\\ 0
			\end{pmatrix},\quad \vartheta_{sh}=\vartheta_{sh}(y), \quad\hbox{with}\quad\varrho_{sh}(y)\vartheta_{sh}(y)=1.
\end{align}
Obviously, when $\mu\neq0,$ due to the strong nonlinear term $|\nabla {\mathbf u}+\nabla {\mathbf u}^\top|^2$ appeared in the third equation of \eqref{mmma}, it is straightforward to verify
that $\vartheta_{sh}(y)$ mush satisfy the following restricted relation:
\begin{align}\label{m3}
\frac{\gamma\mu}{\sigma \Re}\partial_{yy}\vartheta_{sh}(y)=-\frac{\gamma\mu(\gamma-1)M^2}{\Re}.
\end{align}
Solve \eqref{m3}, we can choose $\vartheta_{sh}(y)$ as
\begin{align}\label{m4}
\vartheta_{sh}(y)=\vartheta_r\left[r+(1-r)y-(1-\frac{1}{\vartheta_r})y^2\right]
\end{align}
where $r>0$ is the temperature ratio and $\vartheta_r$ is the recovery temperature defined as follows
\begin{align*}
\vartheta_r \stackrel{\mathrm{def}}{=}1+\frac{(\gamma-1)\sigma M^2}{2}.
\end{align*}
Due to the complicate form of $\vartheta_{sh}(y)$,
 to study the long-time asymptotic behaviour of \eqref{mmma} around the   stationary solution defined in \eqref{m2} and \eqref{m4} is a very difficult problem. To  our best knowledge, there are few results in this direction, see \cite{dole}, \cite{doln}, \cite{abc12+1}, \cite{abc13}, \cite{cha}, \cite{deh94}--\cite{abc19}, \cite{pad}, \cite{abcde11} .

Due to mathematical challenges,
to approach the problem, here, we are only consider a simple case of \eqref{mmma} with the shear viscosity coefficient $\mu=0$ and the  volume viscosity $\nu\neq0$.
 In this case, the system \eqref{mmma} can be rewritten as
\begin{eqnarray}\label{qm6}
\left\{\begin{aligned}
&\varrho_t+{\mathbf{u}}\cdot\nabla {\varrho}+{\varrho}\div{\mathbf{u}}=0,\\
&{{\varrho}}({\mathbf{u}_t}+{\mathbf{u}}\cdot\nabla {\mathbf{u}})+\frac{1}{\gamma{M}^2}\nabla ({\varrho}\vartheta)=\frac{\nu}{\Re}\nabla \div\mathbf{u},\\
&{{\varrho}}({\vartheta_t}+{\mathbf{u}}\cdot\nabla \vartheta)
+ (\gamma-1){\varrho}\vartheta\div {\mathbf{u}}=\frac{\nu\gamma(\gamma-1)M^2}{\Re}|\div {\mathbf u}|^2.
\end{aligned}\right.
\end{eqnarray}
It's straightforward to verify
that
the Couette flow,
\begin{align}\label{qm7}
\varrho_{sh}=1,\quad{\mathbf{u}}_{sh}=\begin{pmatrix}
				y\\ 0
			\end{pmatrix},\quad \vartheta_{sh}=1,
\end{align}
is  a stationary solution of \eqref{qm6}.
 Our goal is to understand the stability and large-time behavior of perturbations near this Couette flow.

Before presenting our main result,
 let us first give a short review of the  extensive mathematical results on  the compressible Navier-Stokes equations.
Glatzel \cite{gla88}, \cite{abc17} studied the linear inviscid and viscous stability properties of the compressible  Couette flow via a normal mode analysis in simplified flow model with constant viscous coefficients and a constant density profile. Duck {\it et al.} \cite{deh94} proved  the linear stability of the plane Couette flow for the non-isentropic compressible Navier-Stokes equations. Chagelishvili {\it et al.} \cite{cha} considered the inviscid stability of the 2D Couette flow. By means of some formal approximation, they showed that the energy of acoustic perturbations grows linear in time due to the transfer of energy from the mean flow to perturbations.
Taking advantage of a fourth-order finite-difference method and a
spectral collocation method, Hu {\it et al.} \cite{abc18} studied the viscous linear stability of supersonic Couette flow for
a perfect gas governed by Sutherland viscosity law.
Kagei \cite{kagg11} proved that the plane Couette flow in an infinite layer is asymptotically stable if the Reynolds and Mach numbers are sufficiently small.  Li {\it et al.} \cite{abc19} investigated
 the stability analysis of the plane Couette flow for the 3D compressible Navier-Stokes equations with Navier-slip boundary condition at the bottom boundary. They shown that the plane Couette flow is asymptotically stable for small perturbation provided that the slip length, Reynolds and Mach numbers satisfy some restricted relation.
 Recently, Antonelli {\it et al.} \cite{dole} studied the linear stability properties of the 2D isentropic
compressible Euler equations linearized around a shear flow given by a monotone profile,
close to the Couette flow, with constant density, in the domain $\mathbb{T}\times \mathbb{R}$.
 Later then, they in \cite{doln} also  studied the linear stability properties of perturbations around the
homogeneous Couette flow for a 2D isentropic inviscid or viscous compressible fluid. Moreover, in the inviscid case, they proved the inviscid damping for the solenoidal
component of the velocity field and  Lyapunov type instability for the density and the
irrotational component of the velocity field. In the viscous case, they  obtained the enhanced dissipation phenomenon.
Zeng {\it et al.} \cite{zenglan} considered the linear stability of the three dimensional isentropic compressible Navier-Stokes equations on $\mathbb{T}\times\mathbb{R}\times\mathbb{T}$. They proved the enhanced dissipation phenomenon and the lift-up phenomenon  around the Couette flow $(y, 0, 0)^\top$.
 The motivation of the present paper is to generalize the results obtained by Antonelli {\it et al.} \cite{dole}, \cite{doln} to the non-isentropic compressible  Navier-Stokes equations with vanished shear viscosity.

Denote
\begin{equation*}
		{\rho}={\varrho}-\varrho_{sh},\qquad
		{\mathbf{v}}={\mathbf{u}}-{\mathbf{u}}_{sh},\qquad{{\theta}}=\vartheta-\vartheta_{sh}.
	\end{equation*}
The linearized system of \eqref{qm6} around the  Couette flow \eqref{qm7} read as follows
\begin{eqnarray}\label{m5}
\left\{\begin{aligned}
&{\partial_t} {\rho} +y{\partial_x}{\rho}+\div {\mathbf{v}}=0,\\
&{\partial_t} {\mathbf{v}}+y{\partial_x} {\mathbf{v}}+\begin{pmatrix}
				v^y\\ 0
			\end{pmatrix}+\frac{1}{\gamma {M}^2} (\nabla{\rho}+\nabla{{\theta}})= \nu\nabla \div\mathbf{v},\\
&{\partial_t} {{\theta}}+y{\partial_x} {{\theta}}+(\gamma-1)\div{\mathbf{v}}=0.
\end{aligned}\right.
\end{eqnarray}
Before going into details of our theorem, we introduce several notations.
Define
\begin{equation*}
		\alpha=\div {\mathbf{v}}, \qquad \omega=\nabla^\perp\cdot {\mathbf{v}}, \quad \hbox{with\quad $\nabla^\perp=(-{\partial_y},{\partial_x})^\top$},
	\end{equation*}
according to
	the Helmholtz projection operators,  we have
	\begin{align}\label{m6}
{\mathbf{v}}=(v^x,v^y)^\top\stackrel{\mathrm{def}}{=}\p [{\mathbf{v}}]+\q [{\mathbf{v}}]
\end{align}
 with
 \begin{align}\label{m7}
\p [{\mathbf{v}}]\stackrel{\mathrm{def}}{=}\nabla^\perp \Delta^{-1}\omega, \qquad\q [{\mathbf{v}}]\stackrel{\mathrm{def}}{=}\nabla \Delta^{-1}\alpha.
\end{align}
From the above definition, one can infer that
\begin{equation}\label{m9}
		v^y={\partial_y}(\Delta^{-1})\alpha+{\partial_x}(\Delta^{-1})\omega,
	\end{equation}
hence, we  can rewrite \eqref{m5}  in terms of $(\rho,\alpha,\omega,{\theta})$ that
\begin{eqnarray}\label{m10}
\left\{\begin{aligned}
&{\partial_t} \rho +y{\partial_x}\rho+\alpha=0,\\
&{\partial_t} \alpha+y{\partial_x} \alpha+2{\partial_x} ({\partial_y}(\Delta^{-1})\alpha+{\partial_x}(\Delta^{-1})\omega)+\frac{1}{\gamma{M}^2}(\Delta\rho+\Delta\theta)=\nu\Delta \alpha,\\
&	{\partial_t} \omega+y{\partial_x} \omega-\alpha=0,\\	
&{\partial_t} {\theta} +y{\partial_x}{\theta}+(\gamma-1)\alpha=0.
\end{aligned}\right.
\end{eqnarray}
Obviously, the above system \eqref{m10} is a closed system regarding of $(\rho,\alpha,\omega,{\theta})$.

Let
	\begin{align*}
		&\widehat{f}(k,\eta)=\frac{1}{2\pi}\iint_{\T\times\mathbb{R}}e^{-i(kx+\eta y)}f(x,y)\,dxdy,
	\end{align*}
and
\begin{align*}
		f(x,y)=\frac{1}{2\pi}\sum_{k}\int_{\mathbb{R}}e^{i(kx+\eta y)}\widehat{f}(k,\eta)\,d\eta,
	\end{align*}
	then we define  $f\in H^{s}(\T\times \R)$ if
	\begin{equation*}
		\norm{f}_{H^{s}}^2=\sum_k\int \langle k,\eta\rangle^{2s} |\hat{f}|^2(k,\eta)\,d\eta< +\infty.
	\end{equation*}

Now, we can state the main result of the present paper.
 \begin{theorem}\label{dingli}
 Let $\gamma>1$, $0<\nu<1$ and  $0<M\le\nu^{-1}$. Assume that $(\rho^{in},\alpha^{in},\omega^{in},\theta^{in})\in H^{\frac{3}{2}}(\T\times\R)$ is the initial data of \eqref{m10} with
\begin{align}\label{lingmo}
\int_\mathbb{T} \rho_{in} \,dx=\int_\mathbb{T} \alpha_{in} \,dx=\int_\mathbb{T} \omega_{in} \,dx=\int_\mathbb{T} {\theta}_{in} \,dx=0.
\end{align}
Then, there exists a positive constant $C$ independent of $\gamma, \nu, M$ such that
 \begin{align*}
			\norm{\p[{\mathbf{v}}]^x(t)}_{L^2}
\le&C\langle t \rangle^{-\frac12}{\gamma^{-1} }\exp(CM(M+1)\nn\\
&\qquad\times\left(\frac{1}{ M}\norm{{\rho}^{in}+{\theta}^{in}}_{H^{\frac32}}+\norm{{\alpha}^{in}}_{H^{\frac32}}
+\gamma\norm{\omega^{in}}_{H^{\frac32}}\right),\nn\\
\norm{\p[{\mathbf{v}}]^y(t)}_{L^2}
\le&C\langle t \rangle^{-\frac32}{\gamma^{-1} }\exp(CM(M+1)\nn\\
&\qquad\times\left(\frac{1}{ M}\norm{{\rho}^{in}+{\theta}^{in}}_{H^{\frac32}}+\norm{{\alpha}^{in}}_{H^{\frac32}}
+\gamma\norm{\omega^{in}}_{H^{\frac32}}\right),
		\end{align*}
and
\begin{align*}
	&\norma{\q[{\mathbf{v}}](t)}{L^2}+\frac{\gamma}{{M}}\norma{\rho(t)}{L^2}+\frac{\gamma}{{M}}\norma{{\theta}(t)}{L^2}
\nonumber\\
&\quad\le C\langle t \rangle^{\frac12}\Bigg\{\left\|\frac{(\gamma-1)\rho^{in}-{\theta}^{in}}{{M}}\right\|_{{L^2}}
+ {(\gamma+1)}\exp(CM(M+1)\nonumber\\
&\qquad\qquad\qquad\qquad
 \times\left(\frac{1}{ M}\norm{{\rho}^{in}+{\theta}^{in}}_{H^{1}}+\norm{{\alpha}^{in}}_{H^{1}}
+\gamma\norm{\omega^{in}}_{H^{1}}\right)\Bigg\}.
\end{align*}
Moreover, if $\rho^{in}, \theta^{in}, \omega^{in}$ additionally satisfy the following relation
	\begin{align}\label{xianzhiguanxi}
\rho^{in}+\gamma\omega^{in}+\theta^{in}=0,
\end{align}
we can obtain the enhanced dissipation for the velocity field
\begin{align*}		
\norm{\p[{\mathbf{v}}]^x(t)}_{L^2}
&\le C\langle t \rangle^{-\frac12}e^{-\frac{1}{16}\nu^{\frac13}t}\exp(CM(M+1) )\left(\norm{{\alpha}^{in}}_{H^{\frac32}}+\frac{1}{ M}\norm{{\omega}^{in}}_{H^{\frac32}}\right),\nn\\
\norma{\p[{\mathbf{v}}]^y(t)}{L^2}
&\le C{{\langle t\rangle}^{-\frac32}}e^{-\frac{1}{16}\nu^{\frac13}t}\exp(CM(M+1) )\left(\norm{{\alpha}^{in}}_{H^{\frac32}}+\frac{1}{ M}\norm{{\omega}^{in}}_{H^{\frac32}}\right),\nn\\
\norma{\q[{\mathbf{v}}](t)}{L^2}+&\frac{1}{{M}}\norma{\rho(t)+{\theta}(t)}{L^2}\nn\\
&\le C\langle t \rangle^{\frac12} e^{-\frac{1}{32}\nu^\frac13 t}(1+\gamma)\exp(CM(M+1) )\left(\norm{{\alpha}^{in}}_{H^{1}}+\frac{1}{ M}\norm{{\omega}^{in}}_{H^{1}}\right).
		\end{align*}
 \end{theorem}
\begin{remark}
At first glance, the enhanced dissipation  phenomenon  of the  velocity  field is some surprising because of there is only dissipation for the compressible part of the velocity. This mainly benefits from the relation \eqref{xianzhiguanxi}  which  gives rise to $\omega=-\frac{1}{\gamma}(\rho+\theta)$. The special  relation
connects compressible and incompressible phenomena. Namely, an increase of the vorticity need to be compensated by a decrease for the density and the temperature.
\end{remark}
\begin{remark}
In \cite{dole},  Antonelli {\it et al.} studied the linear stability properties of the 2D isentropic
compressible Euler equations linearized around a shear flow given by a monotone profile,
close to the Couette flow, with constant density, in the domain $\mathbb{T}\times \mathbb{R}$. For the non-isentropic compressible fluid, how to obtain a similar result is  an interesting problem. This is left in the future work.
\end{remark}

\section{The proof of the main theorem}
\subsection{ Preliminary and the a priori estimates}
First of all, we are concerned with the dynamics of the $x$-averages of the perturbations.
In order to
reveal the distinction between
the zero mode case $k=0$ and the nonzero modes $k\not =0$. We define
\begin{align*}
f_0(y)\stackrel{\mathrm{def}}{=}\frac{1}{2\pi}\int_{\mathbb T}f(x,y)  \,dx, \qquad f_{\not =}(x,y) \stackrel{\mathrm{def}}{=}f(x,y) -f_0(y),
\end{align*}
which represents the projection onto $0$ frequency and the projection onto non-zero frequencies.

Due to the structure of the Couetee flow and the fact that the equations are
	linear, it is clear that the zero mode in $x$ has an independent dynamics with respect to other modes. Consequently, in our analysis we can decouple the evolution of the $k = 0$ mode from the rest of the perturbation.
Integration in $x$ equations in \eqref{m10}, one infer that
\begin{eqnarray}\label{ping1234}
\left\{\begin{aligned}
    &{\partial_t} \rho_0=-\alpha_0,\\
		&{\partial_t} \alpha_0=-\frac{1}{\gamma{M}^2}{\partial_{yy}} \rho_0-\frac{1}{\gamma{M}^2}{\partial_{yy}} {\theta}_0+\nu{\partial_{yy}} \alpha_0,\\
		&{\partial_t} \omega_0=\alpha_0,\\
&{\partial_t} {\theta}_0=-(\gamma-1)\alpha_0.
\end{aligned}\right.
\end{eqnarray}
From the above equation \eqref{ping1234}, we can further get $\alpha_0, \rho_0+{\theta}_0$ satisfy the following damped wave equations:
\begin{equation}\label{ping2+1}
		\partial_{tt} \alpha_0-\nu\partial_{t}{\partial_{yy}} \alpha_0-\frac{1}{{M}^2}{\partial_{yy}}\alpha_0=0, \qquad \text{in }\R,
	\end{equation}
\begin{equation}\label{ping3+1}
		\partial_{tt} (\rho_0+{\theta}_0)-\frac{1}{{M}^2}{\partial_{yy}}(\rho_0+{\theta}_0)=0, \qquad \text{in }\R.
	\end{equation}
Hence, given $\rho_0^{in}=\alpha_0^{in}={\theta}_0^{in}=\omega_0^{in}=0$,  we can get for all $t\ge0$
$$\rho_0(t)=\alpha_0(t)={\theta}_0(t)=\omega_0(t)=0.$$

Consequently, in our analysis we can decouple the evolution of the $k = 0$ mode from the rest of the perturbation.
Let us consider the following coordinate transform
\begin{align*}
\left(
  \begin{array}{c}
    x \\
    y \\
  \end{array}
\right)
\mapsto
\left(
  \begin{array}{c}
    X \\
    Y \\
  \end{array}
  \right)=
  \left(
  \begin{array}{c}
    x-yt \\
    y \\
  \end{array}
  \right).
\end{align*}
Under the new coordinate transform,
 the differential operators change as follows
	\begin{align*}
			{\partial_x}={\partial_X},\quad {\partial_y} = {\partial_Y}-t{\partial_X},\quad
			\Delta= \Delta_L\stackrel{\mathrm{def}}{=}{\partial_{XX}}+({\partial_Y}-t{\partial_X})^2.
	\end{align*}
Define
\begin{align*}
R(t,X,Y)&=\rho(t,X+tY,Y),\quad
			A(t,X,Y)=\alpha(t,X+tY,Y),\\
			\Omega(t,X,Y)&=\omega(t,X+tY,Y),\quad
\Theta(t,X,Y)={\theta}(t,X+tY,Y).
\end{align*}
Then, the linear system \eqref{m10} reduces to the following system in the new coordinates
\begin{eqnarray}\label{hao}
\left\{\begin{aligned}
&{\partial_t} R=-A,\\
&{\partial_t} A=\nu\Delta_LA-2{\partial_X}({\partial_Y}-t{\partial_X})(\Delta_L^{-1})A
-2{\partial_{XX}}(\Delta_L^{-1})\Omega-\frac{1}{\gamma{M}^2}(\Delta_LR+\Delta_L{\Theta}),\\
&{\partial_t} \Omega=A,\\
&{\partial_t} {\Theta}=-(\gamma-1)A.
\end{aligned}\right.
\end{eqnarray}
We want to analyze the system \eqref{hao} on the frequency space, in analogy with respect to
the incompressible Couette flow. So
we define the symbol associated to $-\Delta_L$ as
	\begin{align*}
		p(t,k,\eta)&=k^2+(\eta-kt)^2,
	\end{align*}
and denote the symbol associated to the operator  $2{\partial_X}({\partial_Y}-t{\partial_X})$
	as
	\begin{align*}
		({\partial_t} p)(t,k,\eta)=-2k(\eta-kt).
	\end{align*}
In the moving frame,  for the Laplacian operator, there holds the following inequalities.
\begin{lemma}\label{you14}
Let $p=-\widehat{\Delta}_L=k^2+(\eta-kt)^2$, then for any function $f\in H^{s+2{\beta}}(\mathbb{T}\times\mathbb{R})$, it holds that
	\begin{equation}\label{you15}
	\norma{p^{-{\beta}}f}{H^s}\le C \frac{1}{\langle t \rangle^{2{\beta}}}\norma{f}{{H^{s+2{\beta}}}},\qquad \norma{p^{{\beta}} f}{H^s}\le C \langle t \rangle^{2{\beta}} \norma{f}{{H^{s+2{\beta}}}},
	\end{equation}
	for any ${\beta}>0$.
\end{lemma}
\begin{proof}
	The bound \eqref{you15} follows just by Plancherel Theorem and the basic inequalities for japanese brackets $\langle k,\eta 	\rangle \le C \langle \eta-\xi \rangle \langle k,\xi \rangle$.
\end{proof}

Taking the  Fourier transform of \eqref	{hao} gives rise to
\begin{eqnarray}\label{ahao}
\left\{\begin{aligned}
&{\partial_t}\widehat R=-\widehat A,\\
&{\partial_t} {\widehat{A}}=-\nu p {\widehat{A}}+\frac{{\partial_t} p}{ p}{\widehat{A}}-\frac{2k^2}{ p}\widehat{\Omega}+\frac{ p}{{\gamma M}^2}({\widehat{R}+\widehat{\Theta}}),\\
&{\partial_t} \widehat\Omega=\widehat A,\\
&{\partial_t} \widehat{\Theta}=-(\gamma-1)\widehat A.
\end{aligned}\right.
\end{eqnarray}
To exploit the special structure of the system \eqref{ahao}, we introduce the good unknowns $\Phi$  as
\begin{align}\label{you6}
\Phi=\frac{R+{\Theta}}{\gamma}
\end{align}
from which we can rewrite \eqref{ahao} into
\begin{eqnarray}\label{han}
\left\{\begin{aligned}
&{\partial_t} {\widehat{\Phi}}=-{\widehat{A}},\\
&{\partial_t} {\widehat{A}}=-\nu p {\widehat{A}}+\frac{{\partial_t} p}{ p}{\widehat{A}}-\frac{2k^2}{ p}\widehat{\Omega}+\frac{ p}{{M}^2}{\widehat{\Phi}}.
\end{aligned}\right.
\end{eqnarray}
In order to break through the barrier involve in the term $\Omega$ in \eqref{han}, we deduce from
$${\partial_t}( R +\gamma \Omega+{\Theta})=0
$$
that there holds
\begin{align*}
 R +\gamma \Omega+{\Theta}=R^{in}+\gamma\Omega^{in}+\Theta^{in}.
\end{align*}
Hence, combining with \eqref{you6} leads to
\begin{align}\label{han2}
& \Omega=\Phi^{in}+\Omega^{in}-\Phi.
\end{align}
Substituting \eqref{han2} into \eqref{han},  we get a closed system only involved in $\widehat{\Phi}, \widehat{A}$ other than the initial data
\begin{eqnarray}\label{han3}
\left\{\begin{aligned}
&{\partial_t} {\widehat{\Phi}}=-{\widehat{A}},\\
&{\partial_t} {\widehat{A}}=-\nu p {\widehat{A}}+\frac{{\partial_t} p}{ p}{\widehat{A}}+(\frac{ p}{{M}^2}+\frac{2k^2}{ p}){\widehat{\Phi}}-\frac{2k^2}{ p}(\widehat\Phi^{in}+\widehat\Omega^{in}).
\end{aligned}\right.
\end{eqnarray}
In the following, to obtain the enhanced dissipation, we introduce the ``ghost multiplier'' which has been used in \cite{bed2017}, \cite{bed2019}.

Let
 multiplier $m$ solve the linear ODE for $k \neq 0$:
\begin{align*}
 &\frac{\partial_t{m}}{m} = - \frac{\nu^{1/3}}{\left[ \nu^{1/3} |t-\frac{\eta}{k} | \right]^{2} + 1} \\
&m(0,k,\eta)  = 1.
\end{align*}
Notice that there is a constant $c$ (independent of $k$, $\eta$, $t$, and $\nu$) such that $c < m(t,k,\eta) \leq 1$.
In particular, its presence does not change a norm:
\begin{align}
\norm{m(t,\nabla) \jap{\nabla}^\sigma f}_{L^2} \approx \norm{\jap{\nabla}^\sigma f}_{L^2}. \label{ineq:Mequiv}
\end{align}
The crucial property that $m$ satisfies is:
\begin{align}\label{mxingzhi}
1&\lesssim  \nu^{-1/6}\left(\sqrt{-\frac{\partial_m}{m}(t,k,\eta)}+\nu^{1/2}|k,\eta-kt|\right) \quad \mbox{for } k\neq0,
\end{align}
which implies that
\begin{align}
\norm{f_{\neq}}_{L^2}^2 \lesssim \nu^{-1/3}\left(\norm{ \sqrt{-\frac{\partial_m}{m}} f_{\neq}}_{L^2}^2 + \nu \norm{\nabla_L f_{\neq}}_{L^2}^2\right).
\end{align}

The following  lemma plays a crucial role in our subsequent analysis.
\begin{lemma}\label{keylemma}
	For any $(\rho^{in},\alpha^{in},\omega^{in},\theta^{in}) \in H^{s}(\mathbb{T}\times \mathbb{R})$ with	$s\geq 0$. Assume that  $\gamma>1$, $0<\nu < 1$, and $0<M\leq  \nu^{-1}$. Then there exists a positive constant $C$ independent of $\gamma, \nu, M$ such that
		\begin{align}\label{han307}
&\frac{1}{M}\norm{(p^{-\frac14}{\widehat{\Phi}})(t)}_{H^s}+\norm{(p^{-\frac34}{\widehat{A}})(t)}_{H^s}\nn\\
 &\quad\le C\exp(CM(M+1)\left(\frac{1}{ M}\norm{{\widehat{\Phi}}^{in}}_{H^{s}}+\norm{{\widehat{A}}^{in}}_{H^{s}}
+\norm{{\widehat{\Phi}}^{in}+\widehat{\Omega}^{in}}_{H^{s}}\right).
		\end{align}	
\end{lemma}
	\begin{proof}
For any $s\ge0,$ we define two weighted functions involved in $\widehat{\Phi}$ and $\widehat{A}$ as
\begin{align}\label{han4}
Z_1(t)\stackrel{\mathrm{def}}{=}&\frac{1}{M}\jap{k,\eta}^s(m^{-1} p^{-\frac14}\widehat{\Phi})(t),
\end{align}
\begin{align}\label{han4+1}
 Z_2(t)\stackrel{\mathrm{def}}{=}&\jap{k,\eta}^s(m^{-1} p^{-\frac34}\widehat{A})(t).
\end{align}
To begin with, from equations in \eqref{han3} and definitions of $ Z_1$  and $ Z_2$, a simple computations gives
		\begin{eqnarray}\label{han8}
\left\{\begin{aligned}
&{\partial_t} Z_1=-\frac{{\partial_t} m}{m}Z_1-\frac14\frac{{\partial_t} p}{p}Z_1-\frac{1}{M}p^{\frac12}Z_2,\\
&{\partial_t} Z_2=-\left(\frac{{\partial_t} m}{m}+\nu p\right)Z_2+\frac14\frac{{\partial_t} p}{p}Z_2\\
&\qquad+\left(\frac{1}{M}p^{\frac12}+2M\frac{k^2}{p^{\frac32}}\right)Z_1
-\jap{k,\eta}^s\frac{2m^{-1}k^2}{p^{\frac74}}(\widehat\Phi^{in}+\widehat\Omega^{in}).
\end{aligned}\right.
\end{eqnarray}
Now, we get by multiplying  the first equation by $\bar{Z}_1$ and the second equation by $\bar{Z}_2$ in \eqref{han8} respectively, that
		\begin{equation}\label{han9}
			\frac12{\frac{{d}}{{d} t}} |Z_1|^2=-\frac{{\partial_t} m}{m}|Z_1|^2-\frac14\frac{{\partial_t} p}{p}|Z_1|^2-\frac1M p^{\frac12}\mathrm{Re}(\bar{Z}_1Z_2),
		\end{equation}
	\begin{align}\label{han10}
			\frac12 {\frac{{d}}{{d} t}} |Z_2|^2=&-\left(\frac{{\partial_t} m}{m}+\nu p\right)|Z_2|^2+\frac14 \frac{{\partial_t} p}{p}|Z_2|^2+\frac1M p^{\frac12}\mathrm{Re}(Z_1\bar{Z}_2)\nn\\
&\quad
+{2M\frac{k^2}{p^{\frac32}}\mathrm{Re}(Z_1\bar{Z}_2)}
-\jap{k,\eta}^s\frac{2m^{-1}k^2}{p^{\frac74}}\mathrm{Re}((\widehat\Phi^{in}+\widehat\Omega^{in})\bar{Z}_2).
		\end{align}
			From $p(t,k,\eta)=k^2+(\eta-kt)^2>0$, one has for any
 $t>\eta/k$ there holds ${{\partial_t} p}/{p}>0$, the third term on the right-hand side of the first equation in \eqref{han8} acts as a damping term for $Z_1$. Instead, ${{\partial_t} p}/{p}<0$ for $t<\eta/k$, hence it induces a growth on $Z_1$.
However, the situation is opposite for the second equation involved in $Z_2$. That is to say,
for $t>\eta/k$,  the term $({{\partial_t} p}/{p})Z_2 $ induces a growth, for $t<\eta/k$, the term $({{\partial_t} p}/{p})Z_2 $ acts as  a damping term. Thus, there is a competition between $Z_1$ and $Z_2$.
To balance this relation, we have to consider the time derivative of the mixed terms involved in $Z_1$, $Z_2$:
\begin{align}\label{}
{\frac{{d}}{{d} t}}\left(\frac{{\partial_t} p}{p^{\frac32}} Z_1\right)
=&-\frac{{\partial_t} m}{m}\frac{{\partial_t} p}{p^{\frac32}}Z_1+\left(\frac{2k^2}{p^{\frac32}}-\frac74\frac{({\partial_t} p)^2}{p^{\frac52}}\right)Z_1-\frac{1}{M}\frac{{\partial_t} p}{p}Z_2
	\end{align}
from which and the second equation in \eqref{han8}, we can further get
\begin{align}\label{han11}
\frac{M}{4}{\frac{{d}}{{d} t}} \left(\frac{{\partial_t} p}{p^{\frac32}}\mathrm{Re}(\bar{Z}_1Z_2)\right)
=&-\frac14 \frac{{\partial_t} p}{p}(|Z_2|^2-|Z_1|^2)+{\frac{M}{4}\left(\frac{2k^2}{p^\frac32}-\frac32 \frac{({\partial_t} p)^2}{p^\frac52}\right)\mathrm{Re}(\bar{Z}_1Z_2)}\nn\\
&\quad-{\frac{M}{2}\frac{{\partial_t} m}{m}\frac{{\partial_t} p}{p^{3/2}}\mathrm{Re}(\bar{Z}_1Z_2)}{-\nu\frac{M}{4}\frac{{\partial_t} p}{p^{\frac12}}\mathrm{Re}(\bar{Z}_1Z_2)}+{M^2\frac{k^2{\partial_t} p}{2p^3}|Z_1|^2}\nn\\
&\quad
-\jap{k,\eta}^sm^{-1}\frac{Mk^2\partial_tp}{2p^{\frac{13}{4}}}\mathrm{Re}\left((\widehat\Phi^{in}+\widehat\Omega^{in})\bar{Z}_2\right).
		\end{align}
It's obvious that the first term on the right hand side of \eqref{han11} could cancel two bad  terms $-\frac14\frac{{\partial_t} p}{p}|Z_1|^2$ appeared in \eqref{han9} and  $\frac14\frac{{\partial_t} p}{p}|Z_1|^2$ appeared in \eqref{han10}.

Due to lack of a diffusive term in the equation of $\widehat{\Phi}$, we have to exploit the special structural characteristics (wave structure) of  \eqref{han8} to find  hidden dissipation for $Z_1.$
So, we also need to consider the time derivative of the mixed terms involved in $Z_1$, $Z_2$ with different weight as
\begin{align}\label{}
{\frac{{d}}{{d} t}}\left(p^{-\frac12} Z_1\right)
=&-\frac{{\partial_t} m}{m}p^{-\frac12}Z_1
-\frac{3}{4}\frac{\partial_tp}{p^{\frac32}}Z_1-\frac{1}{M}Z_2
	\end{align}
which combines with the second equation in \eqref{han8} give rise to
\begin{align}\label{han12-1}
- {\frac{{d}}{{d} t}} \left(p^{-\frac12}\mathrm{Re}(\bar{Z}_1Z_2)\right)
=&-\frac{1}{M}\left(1+2M^2\frac{k^2}{p^2}\right)|Z_1|^2
+{\frac{1}{2} \frac{{\partial_t} p}{p^{\frac32}}\mathrm{Re}(\bar{Z}_1Z_2)}\nn\\
			&+{2\frac{{\partial_t} m}{mp^{\frac12}}\mathrm{Re}(\bar{Z}_1Z_2)}+{\frac{1}{M}|Z_2|^2}+{ \nu p^\frac12\mathrm{Re}(\bar{Z}_1Z_2)}\nn\\
&+\jap{k,\eta}^s\frac{2m^{-1}k^2}{p^{\frac94}}\mathrm{Re}((\widehat\Phi^{in}+\widehat\Omega^{in})\bar{Z}_1).
		\end{align}
Finally, in order  to define a coercive energy functional, we need to consider the time derivative of the  term ${{p^{-\frac32}{\partial_t} p}}Z_1$ or ${{p^{-\frac32}{\partial_t} p}}Z_2$. Here, we choose the former.
\begin{equation}\label{han12}
\begin{split}
\frac{M^2}{2}{\frac{{d}}{{d} t}} \left|\frac{{\partial_t} p}{p^{\frac32}}Z_1\right|^2=&{M^2\left(\frac{2k^2{\partial_t} p}{p^3}-\frac74\frac{({\partial_t} p)^3}{p^4}\right)|Z_1|^2}\nn\\
&{-M^2\frac{{\partial_t} m}{m}\frac{({\partial_t} p)^2}{p^{3}}|Z_1|^2}
				{-M \frac{({\partial_t} p)^2}{p^{\frac52}}\mathrm{Re}(\bar{Z}_1Z_2)}.
			\end{split}
		\end{equation}
Now, we define the following energy functional
	\begin{align}\label{han5}
			{\mathcal{E}(t)}= &\frac12\left(1+M^2\frac{({\partial_t} p)^2}{p^3}\right)|Z_1|^2(t)+\frac12|Z_2|^2(t)\nn\\%
		&+\left(\frac{M}{4} \frac{{\partial_t} p}{p^{\frac32}}\mathrm{Re}(\bar{Z}_1Z_2)\right)(t)-\left(\frac{M\nu^{\frac13}}{4} p^{-\frac12}\mathrm{Re}(\bar{Z}_1Z_2)\right)(t).
	\end{align}
Multiplying by $\frac{M\nu^{\frac13}}{4}$ on both hand side of \eqref{han12-1} then summing up \eqref{han9}, \eqref{han10}, \eqref{han11}, and \eqref{han12}
gives
\begin{align}\label{han13}
			&{\frac{{d}}{{d} t}} {\mathcal{E}(t)}+\left(\frac{{\partial_t} m}{m}+\frac{\nu^{\frac13}}{4}\big(1+2M^2\frac{k^2}{p^2}\big)+M^2\frac{{\partial_t} m}{m}\frac{({\partial_t} p)^2}{p^3}\right)|Z_1|^2+\left(\frac{{\partial_t} m}{m}+\nu p\right)|Z_2|^2
\nn\\
&\quad=\frac{\nu^{\frac13}}{4}|Z_2|^2+\frac{M\nu^{\frac43}}{4}  p^{\frac12}\mathrm{Re}(\bar{Z}_1Z_2)-\frac{\nu M}{4}\frac{{\partial_t} p}{p^{\frac12}}\mathrm{Re}(\bar{Z}_1Z_2)+M^2\left(\frac{5k^2{\partial_t} p}{2p^3}-\frac{7}{4}\frac{({\partial_t} p)^3}{p^4}\right)|Z_1|^2\nn\\
&\qquad+ M\left(\frac{\nu^{\frac13}}{8}\frac{{\partial_t} p}{p^\frac32}+\frac52\frac{k^2}{p^{\frac32}}-\frac{11}{8}\frac{({\partial_t} p)^2}{p^{\frac52}}
			+\frac{\nu^{\frac13}}{2} \frac{{\partial_t} m}{mp^{\frac12}}\right)\mathrm{Re}(\bar{Z}_1Z_2)\nn\\
&\qquad
+\jap{k,\eta}^s\frac{M\nu^{\frac13}m^{-1}k^2}{2p^{\frac94}}\mathrm{Re}((\widehat\Phi^{in}+\widehat\Omega^{in})\bar{Z}_1)
-\jap{k,\eta}^s\frac{Mm^{-1}k^2\partial_tp}{2p^{\frac{13}{4}}}\mathrm{Re}((\widehat\Phi^{in}+\widehat\Omega^{in})\bar{Z}_2)\nn\\
&\quad\stackrel{\mathrm{def}}{=}\mathcal{D}_1+\mathcal{D}_2+\mathcal{D}_3+\mathcal{D}_4+\mathcal{D}_5+\mathcal{D}_6+\mathcal{D}_7.
		\end{align}

Now we go to bound the right hand side of \eqref{han13}.
First, from \eqref{mxingzhi}, there holds
\begin{align}\label{ahan13}
\frac{{\partial_t} m}{m}+\nu p\geq \nu^{\frac13}.
\end{align}
Hence, the first  term $\mathcal{D}_1 $ can be absorbed  directly by the left. We next consider $\mathcal{D}_2$.
With the aid of the Cauchy-Schwarz inequality, one has
\begin{align}\label{ahan131}
 |\mathcal{D}_2|
 \leq&  \frac{M\nu^{\frac13}}{8}(\nu|Z_1|^2+\nu p|Z_2|^2)\nn\\
 \leq&  \frac{M\nu}{8}(\nu^{\frac13}|Z_1|^2)+\frac{M\nu^{\frac13}}{8}(\nu p|Z_2|^2).
\end{align}
As a result, in order to absorb  $\mathcal{D}_2$ by the left, we need the assumption $M\nu\le1$.
	Since $|{\partial_t} p|\leq 2|k|p^\frac12$, we can bound the terms $ \mathcal{D}_3$  as follows
\begin{align}\label{han19}
|\mathcal{D}_3|
&\leq\frac{\nu }{4}(\frac{2M|k|p^\frac12}{p^{\frac12}}\mathrm{Re}(\bar{Z}_1Z_2))\nn\\
&\leq\frac{\nu }{4}\left(4\frac{ M^2k^2}{p}|Z_1|^2+\frac14(p|Z_2|^2)\right)\nn\\
&\leq \frac{\nu M^2k^2}{p}|Z_1|^2+\frac{1}{16}\nu p|Z_2|^2.
		\end{align}
In the same manner,  from $|{\partial_t} p|\leq 2|k|p^\frac12$ and the fact that ${|k|}{p^{-\frac32}}\le1$, we have
\begin{align}\label{ahan132}
|\mathcal{D}_4|&\leq {19}M^2\frac{|k|^3}{p^\frac52}|Z_1|^2\leq CM^2\frac{k^2}{p}|Z_1|^2.
\end{align}
In the following, we bound the terms in $\mathcal{D}_5$.
Thanks to $|{\partial_t} p|\leq 2|k|p^\frac12$ again, we have
\begin{align}\label{ahan133}
|\mathcal{D}_5|
\leq& \frac{M\nu^{\frac13}}{4}\frac{|k|}{p^2}\mathrm{Re}(\bar{Z}_1Z_2)+\frac{M}{4}\frac{|k|^2}{p^{\frac32}}\mathrm{Re}(\bar{Z}_1Z_2)
+\frac{M\nu^{\frac13}}{2} \frac{{\partial_t} m}{mp^{\frac12}}\mathrm{Re}(\bar{Z}_1Z_2)\nn\\
\stackrel{\mathrm{def}}{=}&\mathcal{D}_{5,1}+\mathcal{D}_{5,2}+\mathcal{D}_{5,3}.
\end{align}
Due to $\nu\le1$ and $p^{-{\frac32}}\le1$, we can bound $\mathcal{D}_{5,1}$ as
\begin{align}\label{ahan134}
			 |\mathcal{D}_{5,1}|&\leq   CM\frac{k^2}{p}(|Z_1|^2+|Z_2|^2).
		\end{align}
The term  $\mathcal{D}_{5,2}$ can be  controlled similarly if noticing the fact that $|k|p^{-1}\le1$.

For the last term $\mathcal{D}_{5,3}$, we can use $M\nu^{\frac13}\le1$ and $p^{-{\frac12}}\le1$ to get
\begin{align}\label{ahan135}
			 |\mathcal{D}_{5,3}|&\leq   C \frac{{\partial_t} m}{m}(|Z_1|^2+|Z_2|^2).
		\end{align}
Substituting the above estimates involved in $\mathcal{D}_{5,1}, \mathcal{D}_{5,2}, \mathcal{D}_{5,3}$ into \eqref{ahan133}, we can get
\begin{align}\label{ahan136}
			 |\mathcal{D}_{5}|&\leq   CM\frac{k^2}{p}(|Z_1|^2+|Z_2|^2)+ C \frac{{\partial_t} m}{m}(|Z_1|^2+|Z_2|^2).
		\end{align}
From
$\nu\le1$,  $p^{-{1}}\le1$, $|{\partial_t} p|\leq 2|k|p^\frac12$ and
 the multiplier  $m^{-1}$ is a  bound Fourier multiplier, we can get
\begin{align*}
\frac{Mm^{-1}k^2\partial_tp}{2p^{\frac{13}{4}}}+\frac{M\nu^{\frac13}m^{-1}k^2}{2p^{\frac94}}\le CM\frac{|k|^2}{p},
\end{align*}
from which and the Young inequality give rise to
\begin{align}\label{ahan137}
 |\mathcal{D}_6|+|\mathcal{D}_7|
   &\leq CM\frac{k^2}{p}\left(\jap{k,\eta}^{2s } |\widehat\Phi^{in}+\widehat\Omega^{in}|^2 +(|Z_1|^2+|Z_2|^2)\right).
		\end{align}
Noticing the fact that
$$M^2\frac{{\partial_t} m}{m}\frac{({\partial_t} p)^2}{p^3}>0,$$ and then
inserting \eqref{ahan131}, \eqref{han19}, \eqref{ahan132}, \eqref{ahan133},
\eqref{ahan136}, \eqref{ahan137} into \eqref{han13}, we can get
\begin{align}\label{han22}
			&{\frac{{d}}{{d} t}} {\mathcal{E}(t)}+ \frac{\nu^\frac13}{16}\left((1+4M^2\frac{k^2}{p^2})|Z_1|^2+|Z_2|^2\right)\nn\\
&\quad\le CM\frac{k^2}{p}\jap{k,\eta}^{2s } |\widehat\Phi^{in}+\widehat\Omega^{in}|^2+C\left(M(M+1)\frac{k^2}{p}+\frac{{\partial_t} m}{m}\right){\mathcal{E}(t)}.
		\end{align}
As
\begin{align*}
			 4M^2 \frac{k^2}{p^2}\ge M^2\frac{({\partial_t} p)^2}{p^3},
		\end{align*}
we can further get
\begin{equation}\label{han24}
			{\frac{{d}}{{d} t}} {\mathcal{E}(t)} +\frac{\nu^{\frac13}}{16}{\mathcal{E}(t)}\leq CM\frac{k^2}{p}\jap{k,\eta}^{2s } |\widehat\Phi^{in}+\widehat\Omega^{in}|^2+ C\left(M(M+1)\frac{k^2}{p}+2\frac{{\partial_t} m}{m}\right){\mathcal{E}(t)}.
		\end{equation}
 It's easy to check that there holds
		\begin{align*}
			\int_0^t \frac{k^2}{p(\tau)} \,d\tau=\int_0^t\frac{d\tau}{(\frac{\eta}{k}-\tau)^2+1}=\ \left(\arctan(\frac{\eta}{k}-t)-\arctan(\frac{\eta}{k})\right).
		\end{align*}
As a result, applying  Gronwall's inequality to \eqref{han24} we have
		\begin{equation}\label{han25}
			{\mathcal{E}(t)}\leq C\left(\mathcal{E}(0)+\jap{k,\eta}^{2s } |\widehat\Phi^{in}+\widehat\Omega^{in}|^2\right)\exp(CM(M+1)  .
		\end{equation}
From $|{\partial_t} p|<p$, it's not hard to  check that	
	\begin{align*}
		{\mathcal{E}(t)}\approx \ \frac{1}{4}\left((1+M^2\frac{({\partial_t} p)^2}{p^3})|Z_1|^2+|Z_2|^2\right)(t)
	\end{align*}
which combines the fact that
	$m$ is a bounded Fourier multiplier and the definitions of $Z_1, Z_2$, we  can further get
	\begin{align}\label{dengjia}
		\sum_{k}\int {\mathcal{E}(t)}d\eta\approx\frac{1}{M^2}\norm{p^{-\frac14}\widehat{\Phi}(t)}_{H^s}^2+\norm{p^{-\frac34}\widehat{A}(t)}_{H^s}^2
	\end{align}
which implies that
\begin{align}\label{han7}
			&\frac{1}{M}\norm{(p^{-\frac14}{\widehat{\Phi}})(t)}_{H^s}+\norm{(p^{-\frac34}{\widehat{A}})(t)}_{H^s}\nn\\
&\quad\le C\exp(CM(M+1)\left(\frac{1}{ M}\norm{{\widehat{\Phi}}^{in}}_{H^{s}}+\norm{{\widehat{A}}^{in}}_{H^{s}}
+\norm{{\widehat{\Phi}}^{in}+\widehat{\Omega}^{in}}_{H^{s}}\right).
		\end{align}	
This proves the lemma.
	\end{proof}
\subsection{The proof of  Theorem \ref{dingli} for general $\rho^{in}, \theta^{in},\omega^{in}$}
	Thanks to the previous Lemma, we are now to conclude the proof of Theorem \ref{dingli}.
First, from \eqref{han2} and Lemma \ref{keylemma}, we have	
\begin{align}\label{han325}		
			\norm{\Omega(t)}_{H^s}= &\norm{-\Phi(t)}_{H^s}\nn\\
			=&M\norm{p^{\frac14}(M^{-1}p^{-\frac14}\widehat{\Phi})(t)}_{H^s}\nn\\
			\le  &CM\langle t \rangle^{\frac12}\norm{M^{-1}p^{-\frac14}\widehat{\Phi}(t)}_{H^{s+\frac12}}\nn\\
\le &C{\gamma^{-1} }\exp(CM(M+1)\langle t \rangle^{\frac12}C_{{in},{s+\frac12}}
\end{align}
with
\begin{align*}
C_{{in},{s+\frac12}}\stackrel{\mathrm{def}}{=}\left(\frac{1}{ M}\norm{{\rho}^{in}+{\theta}^{in}}_{H^{s+\frac12}}+\norm{{\alpha}^{in}}_{H^{s+\frac12}}
+\gamma\norm{\omega^{in}}_{H^{s+\frac12}}\right).
		\end{align*}
Recall the definition of $\p[{\mathbf{v}}]$ in \eqref{m7}, there holds
\begin{align*}
\norm{\p[{\mathbf{v}}]^x(t)}_{L^2}
=&\norm{\partial_y\Delta^{-1}\omega(t)}_{L^2}\nn\\
=&\norm{(\partial_Y-t\partial_X)(\Delta_L^{-1}\Omega)(t)}_{L^2}\nn\\
			\leq&C\norm{((-\Delta_L)^{-\frac12}\Omega)(t)}_{L^2}.
		\end{align*}
Therefore,  we get from $$p^{\frac12}\jap{kt}\geq C\jap{\eta-kt}\jap{kt}\ge C\jap{\eta}$$ and the estimate \eqref{han325}	 that
		\begin{align}\label{han30}
			\norm{\p[{\mathbf{v}}]^x(t)}_{L^2}\le& C\frac{1}{\langle t\rangle}\norm{{\Omega}(t)}_{H^1}\nn\\
\le&C\langle t \rangle^{-\frac12}{\gamma^{-1} }\exp(CM(M+1)C_{{in},{\frac32}}.
		\end{align}
In the same manner, we can deal with the second component of $\p[{\mathbf{v}}]^y$
\begin{align}\label{han31}
			\norma{\p[{\mathbf{v}}]^y(t)}{L^2}
=&\norma{{\partial_x} \Delta^{-1}\omega}{L^2} =\norm{\partial_X(\Delta_L^{-1}\Omega)(t)}_{L^2}\nn\\
			= & \norm{\frac{k}{p}\Omega(t)}_{L^2}
\le C {{\langle t\rangle}^{-2}}\norm{{\Omega}(t)}_{H^1}\nn\\
\le&C\langle t \rangle^{-\frac32}{\gamma^{-1} }\exp(CM(M+1)C_{{in},{\frac32}}.
		\end{align}

Finally, we estimate the compressible part of the velocity. On the one hand, from the Helmholtz decomposition and the  change of coordinates, we get
\begin{align}\label{han32}
			&\norma{\q[{\mathbf{v}}](t)}{L^2}+\frac{1}{{M}}\norma{\rho(t)+{\theta}(t)}{L^2}\nn\\
&\quad=\norm{(-\Delta)^{-\frac12}\alpha(t)}_{L^2}+\frac{1}{M}\norm{\rho(t)+{\theta}(t)}_{L^2}\nn\\
			&\quad=\norm{(-\Delta_L)^{-\frac12}A(t)}_{L^2}+\frac{1}{M}\norm{R(t)+\Theta(t)}_{L^2}\nn\\
&\quad=\norm{(-\Delta_L)^{-\frac12}A(t)}_{L^2}+\frac{\gamma}{M}\norm{\Phi(t)}_{L^2}.
		\end{align}
As a result, we can further deduce from
 \eqref{han32}, Lemma \ref{keylemma} and the fact that $p\leq\jap{t}^2\jap{k,\eta}^2$ that
		\begin{align}\label{han28}
&\norma{\q[{\mathbf{v}}](t)}{L^2}+\frac{1}{{M}}\norma{\rho(t)+{\theta}(t)}{L^2}
\nn\\
&\quad=\norm{p^{\frac14}(p^{-\frac34}\widehat{A})(t)}_{L^2}
+\frac{\gamma}{M}\norm{p^{\frac14}(p^{-\frac14}\widehat{\Phi})(t)}_{L^2}\nn\\
&\quad\le C \langle t \rangle^{\frac12} \left(\norm{(p^{-\frac34}\widehat{A})(t)}_{H^1}+\frac{\gamma}{M}\norm{(p^{-\frac14}\widehat{\Phi})(t)}_{H^1}\right)\nn\\
			&\quad\le C\langle t \rangle^{\frac12}{(1+\gamma)}{\gamma^{-1} }\exp(CM(M+1)C_{{in},{1}}.
		\end{align}		
On the other hand, by \eqref{m10}, there holds
\begin{align*}
 &({\partial_t}+y{\partial_x}) ((\gamma-1)\rho-{\theta})=0
		\end{align*}
which implies that
\begin{align*}
(\gamma-1)\rho-{\theta}=(\gamma-1)\rho^{in}-{\theta}^{in}.
		\end{align*}
Moreover, we have
\begin{align*}
\left\|\frac{(\gamma-1)\rho(t)-{\theta}(t)}{{M}}\right\|_{{L^2}}^2
=\left\|\frac{(\gamma-1)\rho^{in}-{\theta}^{in}}{{M}}\right\|_{{L^2}}^2.
\end{align*}

A simple computation  gives
\begin{eqnarray*}
\left\{\begin{aligned}
&\frac{\gamma}{{M}}\rho={\frac{(\gamma-1)\rho-{\theta}}{{M}}+\frac{\rho+{\theta}}{{M}}},\\
&\frac{\gamma}{{M}}{\theta}=-\frac{(\gamma-1)\rho-{\theta}}{{M}}+{(\gamma-1)}\frac{\rho+{\theta}}{{M}}.
\end{aligned}\right.
\end{eqnarray*}
Hence
\begin{align*}
\frac{\gamma}{{M}}\norma{\rho(t)}{L^2}
=& \left(\left\|\frac{(\gamma-1)\rho-{\theta}}{{M}}\right\|_{L^2}+\left\|\frac{\rho+{\theta}}{{M}}\right\|_{L^2}\right)\nonumber\\
\le& C\langle t \rangle^{\frac12}\Bigg\{\left\|\frac{(\gamma-1)\rho^{in}-{\theta}^{in}}{{M}}\right\|_{{L^2}}
+  {(1+\gamma)}{\gamma^{-1} }\exp(CM(M+1)C_{{in},{1}}\Bigg\},
\end{align*}
and
\begin{align*}
\frac{\gamma}{{M}}\norma{{\theta}(t)}{L^2}
=& \left\|\frac{(\gamma-1)\rho-{\theta}}{{M}}\right\|_{L^2}+(\gamma-1)\left\|\frac{\rho+{\theta}}{{M}}\right\|_{L^2}\nonumber\\
\le& {C(\gamma-1)}\left\|\frac{(\gamma-1)\rho^{in}-{\theta}^{in}}{{M}}\right\|_{{L^2}}
\nonumber\\
&
+  C\langle t \rangle^{\frac12}{(\gamma-{\gamma^{-1} })}\exp(CM(M+1)C_{{in},{1}}.
\end{align*}
This proves the first case
for general $\rho^{in}, \theta^{in},\omega^{in}$.

\subsection{The proof of  Theorem \ref{dingli} with special $\rho^{in}, \theta^{in},\omega^{in}$ satisfying \eqref{xianzhiguanxi}}
If $$
\rho^{in}+\gamma\omega^{in}+\theta^{in}=0,$$
from
${\partial_t}( R +\gamma \Omega+{\Theta})=0
$
we can infer that
\begin{align*}
& \Omega=-\frac{R+\Theta}{\gamma}=-\Phi.
\end{align*}

Thus, we get a closed system only involved in $\widehat{\Phi}$ and $\widehat{A}$
\begin{eqnarray}\label{aming2}
\left\{\begin{aligned}
&{\partial_t} {\widehat{\Phi}}=-{\widehat{A}},\\
&{\partial_t} {\widehat{A}}=-\nu p {\widehat{A}}+\frac{{\partial_t} p}{ p}{\widehat{A}}+(\frac{ p}{{M}^2}+\frac{2k^2}{ p}){\widehat{\Phi}}.
\end{aligned}\right.
\end{eqnarray}
For the above system \eqref{aming2}, we can make a similar argument as the proof of Lemma \ref{keylemma} to get
another version of \eqref{han24} which don't involve in $\widehat\Phi^{in},\widehat\Omega^{in}$ that
\begin{equation}\label{aming3}
			{\frac{{d}}{{d} t}} {\mathcal{E}(t)} +\frac{\nu^{\frac13}}{16}{\mathcal{E}(t)}\leq C\left(M(M+1)\frac{k^2}{p}+\frac{{\partial_t} m}{m}\right){\mathcal{E}(t)},
		\end{equation}

Consequently, applying  Gronwall's inequality to \eqref{aming3} we have
		\begin{align*}
			{\mathcal{E}(t)}\leq C\exp(CM(M+1) )e^{-\frac{ \nu^\frac13 }{16}t} \mathcal{E}(0)
		\end{align*}
which combines with the equivalent  relation \eqref{dengjia} give
	\begin{align}\label{aming5}
			&\frac{1}{M}\norm{(p^{-\frac14}{\widehat{\Phi}})(t)}_{H^s}+\norm{(p^{-\frac34}{\widehat{A}})(t)}_{H^s}\nn\\
		&\quad\le C\exp(CM(M+1) ) e^{-\frac{1}{32}\nu^\frac13 t}\left(\frac{1}{\gamma M}\norm{{\Phi}^{in}}_{H^{s}}+\norm{{\alpha}^{in}}_{H^{s}}\right).
		\end{align}		
With \eqref{aming5} in hand, we can follow the same argument as the derivation of \eqref{han30}, \eqref{han31}, \eqref{han28} to get
\begin{align*}
			\norm{\p[{\mathbf{v}}]^x(t)}_{L^2}\le &C\exp(CM(M+1) )\langle t \rangle^{-\frac12}e^{-\frac{1}{16}\nu^{\frac13}t}\left(\frac{1}{\gamma M}\norm{{\rho}^{in}+{\theta}^{in}}_{H^{\frac32}}+\norm{{\alpha}^{in}}_{H^{\frac32}}\right),
		\end{align*}
\begin{align*}
			\norma{\p[{\mathbf{v}}]^y(t)}{L^2}
\le &C\exp(CM(M+1) ){{\langle t\rangle}^{-\frac32}}e^{-\frac{1}{16}\nu^{\frac13}t}\left(\frac{1}{\gamma M}\norm{{\rho}^{in}+{\theta}^{in}}_{H^{\frac32}}+\norm{{\alpha}^{in}}_{H^{\frac32}}\right),
		\end{align*}
and
\begin{align*}
&\norma{\q[{\mathbf{v}}](t)}{L^2}+\frac{1}{{M}}\norma{\rho(t)+{\theta}(t)}{L^2}
\nn\\
&\quad\le C(1+\gamma)\exp(CM(M+1) )\langle t \rangle^{\frac12} e^{-\frac{1}{32}\nu^\frac13 t}\left(\frac{1}{\gamma M}\norm{{\rho}^{in}+{\theta}^{in}}_{H^{1}}+\norm{{\alpha}^{in}}_{H^{1}}\right).
		\end{align*}
The proof of Theorem \ref{dingli} is completed.

 \textbf{Acknowledgement.}
This work is supported by  NSFC under grant number 11601533.

\end{document}